\documentclass[12 pt]{article}
\usepackage[centertags]{amsmath}
\usepackage{amsfonts}
\usepackage{amssymb}
\usepackage{amsthm}
\usepackage{newlfont}
\usepackage{graphicx}
\usepackage{color}

\newtheorem{theor}{Theorem}
\newtheorem{pr}{Proposition}
\newtheorem{cl}{Corollary}
\newtheorem{lm}{Lemma}
\theoremstyle{definition}
\newtheorem{defn}{Definition}
\newtheorem{ex}{Example}
\theoremstyle{remark}

\numberwithin{equation}{section}

\begin{document}
\title{On extended quasi-Armendariz rings}
\author{\small{Hamideh Pourtaherian$^{1}$} ,
Isamiddin S.Rakhimov$^{2}$ \thanks{%
Corresponding author.\newline \textit{E-mail addresses:\small{
$^1$pourtaherian09@gmail.com, $^2$risamiddin@gmail.com}}}}
\date{}
\maketitle
\begin{center}
\scriptsize $^{1,2}$ Department of Mathematics, FS and
$^2$Institute for Mathematical Research (INSPEM)\\ UPM, 43400,
Serdang, Selangor Darul Ehsan, Malaysia.
\end{center}

\maketitle

\begin{abstract}
This paper deals with the quasi-Armendariz ring in general
setting. We generalize the notions of quasi-Armendariz and
$\alpha$-skew Armendariz ring to quasi $\alpha$-Armendariz and
quasi $\alpha$-skew Armendariz ring and investigate their
properties. We also consider the quasi Armendariz properties of
Laurent type rings.
\end{abstract}

\section{Introduction}

Let $R$ be a ring and $\alpha :R\longrightarrow R$ be an
endomorphism. Then $\alpha $-derivation $\delta $ of $R$ is an
additive map such that $\delta (ab)=\delta (a)b+\alpha (a)\delta
(b),$ for all $a,b\in R.$ \emph{The Ore extension} $R[x;\alpha
,\delta ]$ of $R$ is
the ring with the new multiplication $%
xr=\alpha (r)x+\delta (r)$ in the polynomial ring over $R,$
where $r\in R.$ If $\delta =0,$ we write $%
R[x;\alpha ]$ and it is said to be a skew polynomial ring (also
\emph{The Ore extension of endomorphism type}.)

Some properties of skew polynomial rings have been studied in
 \cite{shabih}, \cite{cortes}, \cite{com}, \cite{skew} and \cite{a}.

A ring $R$ is called \emph{Armendariz}, if whenever polynomials
$f\left( x\right) ~=~a_{0}+a_{1}x~+~...~+~a_{n}x^{n}, $ $\
g\left(x\right) =b_{0}+b_{1}x+...+b_{m}x^{m}\in R\left[ x\right]$
satisfy $f\left( x\right) g\left( x\right) =0,$ then
$a_{i}b_{j}=0$ for each $i$ and $j.$ (The converse is always
true). The name ``Armendariz'' was used by them because E.Armendariz
\cite{ep} had shown that every reduced ring (i.e., has no nonzero
nilpotent elements) satisfies this condition. Further properties
of Armendariz rings and related problems have been studied in
\cite{an}, \cite{hirano}, \cite{lee}, \cite{h} and \cite{rege}.

According to Krempa \cite{k}, an endomorphism $\alpha $ of a ring
$R$ is called rigid, if for $r\in R$ the condition $r\alpha (r)=0$
implies $r=0$ . In \cite{rigid}, a ring $R$ has been called
$\alpha $-rigid if there exists a rigid endomorphism $\alpha $ of
$R.$ In the same paper it has been shown also that any rigid
endomorphism of a ring is a monomorphism and $\alpha $-rigid rings
are reduced .

The concept of $\alpha $-skew Armendariz ring has been introduced
in \cite{skew} which is a generalization of $\alpha $-rigid ring
and Armendariz ring. A ring $R$ is said to be $\alpha $-skew
Armendariz ring, if for $p=\sum \limits_{i=0}^{m}a_{i}x^{i}$ and
$q=\sum \limits_{j=0}^{n}b_{j}x^{j}$ in $R[x;\alpha ]$ the
condition $pq=0$ implies $a_{i}\alpha^{i}(b_{j})=0$ for all $i$
and $j.$

The Armendariz property of rings was extended to skew polynomial
rings in \cite{a}. Following Hong et al \cite{a}, a ring $R$ is
called $\alpha$-Armendariz if for $p=\sum
\limits_{i=0}^{m}a_{i}x^{i}$ and $q~=~\sum
\limits_{j=0}^{n}b_{j}x^{j}$ in $R[x;\alpha ]$ the condition
$pq=0$ implies $a_{i}b_{j}=0$ for all $i$ and $j.$ $\alpha$-
Armendariz ring is a generalization of Armnedariz ring and
$\alpha$- rigid ring. The authors of \cite{a} proved that an
$\alpha$- Armendariz ring is $\alpha$-skew Armendariz.

In this paper we introduce the notion of quasi $\alpha$-
Armendariz ring, which is a generalization of $\alpha$- rigid and
quasi-Armendariz rings ($\alpha=I_{R}$), by considering the skew
polynomial ring $R[x;\alpha ]$ in place of the ring $R[x].$
Motivated by result of \cite{hirano}, \cite{skew}, we investigate
a generalization of $\alpha $-skew Armendariz rings, which we call
quasi $\alpha $-skew Armendariz rings. Here, we extend the results
of \cite{hirano}, \cite{skew}, \cite{rigid}, \cite{a} to quasi
$\alpha $- Armendariz (or simply $q.\alpha$-Armendariz) and quasi
$\alpha $-skew Armendariz rings (or simply $q.\alpha$-skew
Armendariz). The notions of quasi $\alpha$-Armendariz and quasi
$\alpha$-skew Armendariz rings are useful in understanding the
relationships between annihilators of rings $R$ and $R[x;\alpha
].$ We also introduce the notion of Laurent $q.\alpha$-skew
Armendariz ring and $q.\alpha$-skew power series Armendariz ring
of Laurent type ring $R$ with respect to an automorphism $\alpha$
of $R. $ We do this by considering the quasi-Armendariz condition
on polynomials in $R[x, x^{-1};\alpha ]$ and $R[[x, x^{-1}; \alpha
]]$ instead of $R[x;\alpha ]. $ This provides an opportunity to
study qausi-Armendariz rings in more general setting. The quasi
$\alpha$-skew Armendariz property was also studied under the name
``\emph{$\alpha$-skew quasi-Armendariz}'', by Hong et al.
\cite{q}. The conditions (SQA1)-(SQA4) in \cite{com} are skew
polynomial, skew power series, skew Laurent polynomial and skew
Laurent power series versions of quasi-Armendariz rings. Cortes
\cite{cortes} used the term \emph{quasi-skew Armendariz} for what
is called
\emph{$\alpha$-skew quasi-Armendariz} when $\alpha$ is an automorphism.\\
In \cite{shabih} Ba\c{s}er and Kwak introduced the concept of
$\alpha$-quasi-Armendariz ring. A ring $R$ is called
\emph{quasi-Armendariz ring with the endomorphism $\alpha$}(or
simply \emph{$\alpha$-quasi-Armendariz}) if for
$p(x)=a_{0}+a_{1}x+...+a_{m}x^{m}, \ \ q(x)
=b_{0}+b_{1}x+...+b_{n}x^{n}$ in $R[x;\alpha ]$ satisfy $p(x)
R[x;\alpha]q(x)=0, $ implies  $a_{i}R[x; \alpha]b_{j}=0$ for all
$0\leq i\leq m$ and $0\leq j\leq n$ or equivalently, $a_{i}R
\alpha^{t}(b_{j})=0$ for any nonnegative integer $t$ and all $i,
j$ \cite{shabih}. Ba\c{s}er and Kwak \cite{shabih} also showed
that every $\alpha$-quasi-Armendariz ring is $\alpha$-skew
quasi-Armendariz in case that $\alpha$ is an epimorphism; but the
converse does not hold, in general.

The organization of the paper is as follows. First, we introduce
quasi $\alpha $- Armendariz and quasi $\alpha $-skew Armendariz
rings and investigate their properties (Section 2). Then in
Section 3 we consider Laurent $q.\alpha $-skew Armendariz ring and
Laurent $q.\alpha$-skew Armendariz power series.

Onward throughout the paper, $\alpha $ stands for an
 endomorphism of $R, $ unless
specially noted. The set of all integer numbers and the set of all
rational numbers are denoted by $\mathbb{Z}$ and $\mathbb{Q}$,
respectively. We keep the standard notation $\mathbb{Z}_{4}$ for
the ring of integers module $4$.

\section{Extensions of quasi-Armendariz property}
In this section, we introduce two new classes of rings, which are
generalizations of quasi-Armendariz rings, replacing $R[x]$ by the
skew polynomial ring $R[x;\alpha ]$. We begin with the following
definition.
\begin{defn}
A ring $R$ is called a \emph{quasi $\alpha$-Armendariz ring} (or
simply \emph{$q.\alpha$-Armendariz}) ring if whenever $p=\sum
\limits_{i=0}^{m}a_{i}x^{i}$ and $q=\sum
\limits_{j=0}^{n}b_{j}x^{j}$ in $R[x;\alpha ]$ satisfy
$pR[x;\alpha]q=0, $ we have $a_{i}R b_{j}=0$ for each $i$ and $j.$
\end{defn}

It is easily  to see that an $\alpha$-rigid ring is
$q.\alpha$-Armendariz.

\begin{defn}
A ring $R$ is called a \emph{quasi $\alpha$-skew Armendariz ring}
(or simply \emph{$q.\alpha$-skew Armendariz}) if whenever
$p=a_{0}+a_{1}x+...+a_{m}x^{m}, \ \ q
=b_{0}+b_{1}x+...+b_{n}x^{n}$ in $R[x;\alpha ]$ satisfy
$pR[x;\alpha]q=0, $ implies  $a_{i}R \alpha^{i}(b_{j})=0$ for all
$0\leq i\leq m$ and $0\leq j\leq n.$
\end{defn}
Recently, Hong et al. (see \cite{q}) proved that for reduced ring
$R$ with an endomorphism $\alpha$ and $p=\sum
\limits_{i=0}^{m}a_{i}x^{i}$ and $q=\sum
\limits_{j=0}^{n}b_{j}x^{j} \in R[x;\alpha ]$ the following holds
true:

 $pR[x;\alpha]q=0$ if and only if $a_{i}R
\alpha^{i+t}(b_{j})=0$ for any integer $t\geq 0$ and $i, j.$ By
using the result of Hong et al. we show that, reduced rings are
$q.\alpha$-skew Armendariz.
\begin{pr} A reduced ring with an endomorphism $\alpha$ is $q.\alpha$-skew
Armendariz.
\end{pr}
\begin{proof} Let $R$ be a reduced ring with an endomorphism
$\alpha$ and $p=\sum \limits_{i=0}^{m}a_{i}x^{i}$ and $q=\sum
\limits_{j=0}^{n}b_{j}x^{j}$ be elements of $R[x;\alpha ]$ such
that $pR[x;\alpha]q=0. $ Then for any $r \in R$ we have,
\begin{equation}
(a_{0}+a_{1}x+...+a_{m}x^{m})r (b_{0}+b_{1}x+...+b_{n}x^{n})=0
\end{equation}
We claim that $a_{i}R \alpha^{i}(b_{j})=0$ for any $i, j.$ We
proceed by induction on $i+j. $ If $i+j=0, $ then
$a_{0}R(b_{0})=0. $ Assume that for $i+j=k-1$ where $1\leq k\leq
m+n, $ $a_{i}R \alpha^{i}(b_{j})=0$ is true. The coefficient of
$x^{k}$ is zero, so we obtain:
\begin{equation}
a_{0}rb_{k}+a_{1}\alpha(rb_{k-1})+...+a_{k}\alpha^{k}(rb_{0})=0
\end{equation}
We replace $r$ by $b_{0}$ in (2.2) we get
\begin{equation}
a_{0}b_{0}b_{k}+a_{1}\alpha(b_{0}b_{k-1})+...+a_{k}\alpha^{k}(b_{0}^{2})=0
\end{equation}
Thus $a_{k}\alpha^{k}(b_{0})\alpha^{k}(b_{0})=0$ by the induction
hypothesis. Since $R$ is reduced we get
$\alpha^{k}(b_{0})a_{k}\alpha^{k}(b_{0})=0. $ Therefore,
$(\alpha^{k}(b_{0})a_{k})^{2}=0, $ and $\alpha^{k}(b_{0})a_{k}=0,$
then $a_{k} R \alpha^{k}(b_{0})=0. $ Hence (2.2) becomes
\begin{equation}
a_{0}rb_{k}+a_{1}\alpha(rb_{k-1})+...+a_{k-1}\alpha^{k-1}(rb_{1})=0
\end{equation}
Next We replace $r$ by $b_{1}$ in (2.4), then $a_{k-1}
\alpha^{k-1}(b_{1})=0, $ so $a_{k-1} R \alpha^{k-1}(b_{1})=0. $
With the same method as above we have, $a_{i}R \alpha^{i}(b_{j})=0
$ for any $i+j=k. $ Consequently we have $a_{i}R
\alpha^{i}(b_{j})=0$ for any $0\leq i\leq m$ and $0\leq j\leq n.$
Therefore $R$ is $q.\alpha$-skew Armendariz ring.
\end{proof}
\begin{cl}
$\alpha$-rigid rings are $q.\alpha$-skew Armendariz.
\end{cl}
Here there is an example of $q.\alpha$-skew Armendariz ring which
is not $\alpha$-skew Armendariz.
\begin{ex}
We refer to the example from \cite{skew}, Let
$R=\mathbb{Z}_{2}\oplus \mathbb{Z}_{2}$ with endomorphism $\alpha
:R\longrightarrow R$ defined by $\alpha((a,b))=(b,a). $ The ring
$R$ is reduced, so it is $q.\alpha$-skew Armendariz ring but is
not $\alpha$-skew Armendariz.
\end{ex}

 Ba\c{s}er et al. (see \cite{mccoy}) proved that for a ring isomorphism
$\sigma :R\longrightarrow S,$ the ring $R$ is an $\alpha$-skew
McCoy if and only if $S$ is a $(\sigma \alpha \sigma^{-1})$-skew
McCoy ring. The next theorem is a counterpart of this result in
$q.\alpha$-Armendariz and $q.\alpha$-skew Armendariz rings cases.
\begin{theor}
Let $\sigma$ be an isomorphism from a ring $R$ into a ring $S.$
Then one has
\begin{enumerate}
 \item $R$ is an $\alpha$-skew Armendariz ring if and only if $S$
 is a $(\sigma \alpha \sigma^{-1})$-skew Armendariz.
\item $R$ is an $\alpha$-Armendariz ring if and only if $S$
 is a $(\sigma \alpha \sigma^{-1})$-Armendariz.
 \item $R$ is a $q.\alpha$-Armendariz ring if and only if $S$
 is a $q.(\sigma \alpha \sigma^{-1})$-Armendariz.
 \item $R$ is a $q.\alpha$-skew Armendariz ring if and only if $S$
 is a $q.(\sigma \alpha \sigma^{-1})$-skew Armendariz.
\end{enumerate}
\end{theor}
\begin{proof}
Let $a'$ stand for the image of an element $a$ of $R$ under
$\sigma.$ Then one has
$$p=\sum \limits_{i=0}^{m}a_{i}x^{i}, \ q=\sum
\limits_{j=0}^{n}b_{j}x^{j}\in R[x;\alpha ] \ \ \mbox{if and only
if} \ \ p'=\sum \limits_{i=0}^{m}a'_{i}x^{i}, \ q'=\sum
\limits_{j=0}^{n}b'_{j}x^{j} \in S[x;\sigma \alpha \sigma^{-1} ].
$$

First we prove that $$pR[x;\alpha]q=0 \ \ \mbox{if and only if} \
\ p'S[x;\sigma \alpha \sigma^{-1}]q'=0.$$ Indeed, suppose that $p
R[x;\alpha ]q=0. $ Let $h=\sum \limits_{l=0}^{t}c_{l}x^{l}\in
R[x;\alpha ]$ and $h'=\sum
\limits_{l=0}^{t}c'_{l}x^{l}$ be elements of ~$S[x;\sigma \alpha \sigma^{-1} ]. $\\
Then $phq=0$ if and only if

$\sum \limits_{i+j+l=k}a_{i}\alpha^{i}(c_{l})\alpha^{i}(b_{j})=0$
for each $0\leq k\leq m+n+t, $

if and only if $\sum
\limits_{i+j+l=k}\sigma(a_{i})\sigma(\alpha^{i}(c_{l}))\sigma(\alpha^{i}(b_{j}))=\sum
\limits_{i+j+l=k}\sigma(a_{i}) (\sigma \alpha
\sigma^{-1})^{i}\sigma(c_{l})(\sigma \alpha
\sigma^{-1})^{i}\sigma(b_{j})=\\  \sum \limits_{i+j+l=k}a'_{i}
(\sigma \alpha \sigma^{-1})^{i}(c'_{l})(\sigma \alpha
\sigma^{-1})^{i}(b'_{j})=0,$

for each $0\leq k\leq m+n+t, $ ($(\sigma \alpha
\sigma^{-1})^{i}=\sigma \alpha^{i} \sigma^{-1}$ for any integer
$i$) if and only if $p'h'q'=0.$ Since $h$ runs $R[x;\alpha ]$  its
image $h'$ runs $S[x;\sigma\alpha \sigma^{-1}].$ Therefore we get
$p'S[x;\sigma \alpha \sigma^{-1}]q'=0.$

Now we turn to the assertions of the theorem.
\begin{enumerate}
\item Suppose that $R$ is $\alpha$-skew Armendariz. Let $p'q'=0,$
then according to \cite{mccoy} we get $pq=0. $ Hence $a_{i}
\alpha^{i}(b_{j})=0$ for each $0\leq i\leq m, 0 \leq j \leq n, $
due to $R$ is $\alpha$-skew Armendariz. Since $\sigma$ is
monomorphism, we obtain $\sigma(a_{i}
\alpha^{i}(b_{j}))=\sigma(a_{i})(\sigma \alpha
\sigma^{-1})^{i}\sigma(b_{j})=a'_{i}(\sigma \alpha
\sigma^{-1})^{i}b'_{j}=0$ for all $i$ and $j.$ Hence $S$ is
$(\sigma \alpha \sigma^{-1})$-skew Armendariz ring. The converse
is proved similarly. \item Assume that $ R$ is $\alpha$-Armendariz
ring. Let $p'q'=0, $ then $pq=0. $ Hence $a_{i}b_{j}=0$ for all
$0\leq i\leq m, 0 \leq j \leq n. $ If apply $\sigma$ we get
$\sigma(a_{i}b_{j})=\sigma(a_{i})\sigma(b_{j})=a'_{i}b'_{j}=0$ for
all $i$ and $j. $ Therefore $S$ is $(\sigma \alpha
\sigma^{-1})$-Armendariz ring. The converse is similar.\item It is
proved by the same method as the above.
  \item Suppose that $R$ is
$q.\alpha$-skew Armendariz. Let $p'S[x; \sigma\alpha
\sigma^{-1}]q'=0, $ then we have $pR[x;\alpha]q=0, $ and $a_{i} R
\alpha^{i}(b_{j})=0$ for each $0\leq i\leq m, 0 \leq j \leq n. $
If applying $\sigma$ we get $\sigma(a_{i})\sigma(R) (\sigma \alpha
\sigma^{-1})^{i} \sigma(b_{j})=a'_{i} S (\sigma \alpha
\sigma^{-1})^{i} b'_{j}=0, $ for all $i$ and $j. $ Therefore $S$
is $q.(\sigma \alpha \sigma^{-1})$-skew Armendariz ring.
 The "only if part" again is similar.
\end{enumerate}
\end{proof}
\begin{pr}
Let $R_1$ and $R_2$ be rings. Let $\alpha_1$ and $\alpha_2$ be
endomorphisms of $R_1$ and $R_2,$ respectively. Consider the
endomorphism $\alpha=(\alpha_1,\alpha_2)$ of the direct product
$R=R_1\oplus R_2$ defined by
$\alpha(x_1,x_2)=(\alpha_1(x_1),\alpha_2(x_2)).$ Then
\begin{enumerate}
\item If $R_{1}$ and $R_{2}$ are $\alpha_{1}$-Armendariz and
$\alpha_{2}$-Armendariz, respectively, then $R$ is an
$\alpha$-Armendariz ring. \item If $R_{1}$ and $R_{2}$ are
$\alpha_{1}$-skew Armendariz and $\alpha_{2}$-skew Armendariz
ring, respectively, then $R$ is an $\alpha$-skew Armendariz ring.
\end{enumerate}
\end{pr}
\begin{proof}
Let $p=\sum \limits_{i=0}^{m}(a_{i},b_{i})x^{i}, \ q=\sum
\limits_{j=0}^{n}(c_{j},d_{j})x^{j}$ be elements of $R[x;\alpha ]$
with $pq=0. $ Let $p=(p_{0}, p_{1}), \ q~=~(q_{0}, q_{1}), $ where
$p_{0}=\sum \limits_{i=0}^{m}a_{i}x^{i}, \ q_{0}=\sum
\limits_{j=0}^{n}c_{j}x^{j} \in R_{1}[x;\alpha_{1} ]$ and
$p_{1}=\sum \limits_{i=0}^{m}b_{i}x^{i}, \ q_{1}=\sum
\limits_{j=0}^{n}d_{j}x^{j} \in R_{2}[x;\alpha_{2} ].$
\begin{enumerate}
\item Suppose that $pq=(p_{0}, p_{1})(q_{0}, q_{1})=(p_{0}q_{0},
p_{1}q_{1})=0.$ Then $p_{0}q_{0}=0$ and $p_{1}q_{1}=0.$ However,
by the hypotheses of the theorem $a_{i}c_{j}=0$ and
$b_{i}d_{j}=0.$ Therefore, $(a_{i}, b_{i})(c_{j},
d_{j})=(a_{i}c_{j}, b_{i}d_{j})=0$ and it implies that $R$ is an
$\alpha$-Armendariz ring. \item Like to the previous case
$pq=(p_{0}, p_{1})(q_{0}, q_{1})=(p_{0}q_{0}, p_{1}q_{1})=0$
implies $p_{0}q_{0}=0,$ and $p_{1}q_{1}=0. $ According to the
assumptions on $R_{1}$ and $R_{2}$ we get
$a_{i}\alpha^{i}_{1}(c_{j})=0, $ and
$b_{i}\alpha^{i}_{2}(d_{j})=0.$ Then $(a_{i},
b_{i})\alpha^{i}((c_{j}, d_{j}))=(a_{i}\alpha^{i}_{1}(c_{j}),
b_{i}\alpha^{i}_{2}(d_{j}))=0,$ therefore $R$ is an $\alpha$-skew
Armendariz ring.
\end{enumerate}
\end{proof}

It is obvious that the assertions of the previous proposition can
be easily extended to finite number rings case.

 Recall that, a ring is said to be \emph{semi-commutative}, if whenever elements
$a,b\in R$ satisfy $a b=0$, then $aRb=0. $
\begin{pr}
Let $R$ be a semi-commutative  $\alpha$-skew Armendariz ring and
$\alpha(1)=1. $ Then $R$ is $q.\alpha$-skew Armendariz ring.
\end{pr}
\begin{proof}
Let $p,q$ be elements of $R[x;\alpha], $ with $pR[x; \alpha]q=0. $
Since $R[x;\alpha]$ is a ring with identity, we get $pq=0. $ Hence
$a_{i}\alpha^{i}(b_{j})=0. $ By the hypothesis, $R$ is
semi-commutative, then $a_{i}R \alpha^{i}(b_{j})=0. $ Therefore
$R$ is $q.\alpha$-skew Armendariz ring.
\end{proof}

 Since for $\alpha$-Armendariz rings $\alpha(1)=1 $  (see \cite{a}) then as a consequence of Proposition 3 one can say that a semi-commutative
$\alpha$-Armendariz ring is $q.\alpha$-skew Armendariz. It is also
easily can be seen that a semi-commutative $\alpha$-Armendariz
ring is $q.\alpha$-Armendariz.

Recall that a ring $R$ is said to be
 \emph{symmetric}, if $abc=0$ $\implies $
$bac=0$ for all $a,b,c\in R,$ \emph{reversible}, if $ab=0$ implies
$ba=0$ for all $a,b\in R.$
\begin{cl}
 Let $R$ be a semi-commutative (in particular, reduced, commutative,
symmetric, reversible) ring. If $R$ is $\alpha$-skew Armendariz
and $\alpha(1)=1$ then $R$ is $q.\alpha$-skew Armendariz ring.
\end{cl}
\begin{proof} The proof is immediate from Proposition 3.
\end{proof}
The following proposition is an extension of a result from
\cite{skew}.

\begin{pr}
Let $R$ be a domain and $\alpha$ be a monomorphism of $R. $ Then
$R$ is $q.\alpha$-skew Armendariz ring.
\end{pr}
\begin{proof}
Let $p=a_{0}+a_{1}x+...+a_{m}x^{m},q =b_{0}+b_{1}x+...+b_{n}x^{n}$
be elements of $R[x;\alpha ], $ with $pR[x;\alpha]q=0. $ We claim
that $a_{i}R \alpha^{i}(b_{j})=0, $ for all $i$ and $j.$ Assume
that there exists $a_{k}\neq 0$ such that
$a_{0}=a_{1}=...=a_{k-1}=0, $ where $0\leq k\leq m. $ Let $c$ be a
nonzero arbitrary element of $R, $ hence $pcq=0. $ The coefficient
of $x^{k}$ in $pcq$ is $a_{0}c b_{k}+a_{1} \alpha(c)
\alpha(b_{k-1})+...+a_{k}\alpha^{k}(c)\alpha^{k}(b_{0}). $ It is
zero. Since $R$ is domain and $\alpha$ is monomorphism, we obtain
$a_{k}\alpha^{k}(c)\alpha^{k}(b_{0})=0$ and so
$\alpha^{k}(b_{0})=0.$\\
 Following the same
way as above it is easy to see that $\alpha^{k}(b_{1})=0. $
Continuing this process we derive,
$\alpha^{k}(b_{0})=\alpha^{k}(b_{1})=...=\alpha^{k}(b_{n})=0. $ So
$a_{i}R \alpha^{i}(b_{j})=0, $ for any $0\leq i\leq m, 0\leq j\leq
n. $ Therefore $R$ is $q.\alpha$-skew Armendariz ring.
\end{proof}

The proof of the following proposition can be  handled similarly.
\begin{pr}
Let $R$ be a domain and $\alpha$ be a monomorphism of $R. $ Then
$R$ is $q.\alpha$-Armendariz ring.
\end{pr}
\begin{lm}
Let $R$ be an $\alpha$-skew Armendariz ring. If $P_{i_{1}},...,
P_{i_{n}}\in R[x;\alpha ]$ with $P_{i_{1}}P_{i_{2}}...P_{i_{n}}=0,
$ then
$a_{i_{1}}\alpha^{i_{1}}(a_{i_{2}})\alpha^{i_{1}+i_{2}}(a_{i_{3}})
... \alpha^{i_{1}+...+i_{n-1}}(a_{i_{n}})=0, $ where $a_{i_{j}}$
is an appropriate coefficient of $P_{i_{j}}$ for each $j$.
\end{lm}
\begin{proof}
We prove the lemma by induction on $n. $  If $n=1$ then from
$P_{i_{1}}=0$ we obtain $a_{i_{1}}=0. $  For $n=2$ by the
definition of $\alpha$-skew Armendariz we have
$a_{i_{1}}a_{i_{2}}=0. $ Suppose that our claim
is true for $n-1$ we prove for $n. $\\
Let $(P_{i_{1}}... P_{i_{n-1}})P_{i_{n}}=0, $ by the induction
hypothesis and using the definition of $\alpha$-skew Armendariz
ring, we have $(a_{i_{1}}\alpha^{i_{1}}(a_{i_{2}}) ...
\alpha^{i_{1}+...+i_{n-2}}(a_{i_{n-1}}))
\alpha^{i_{1}+...+i_{n-1}}(a_{i_{n}})=0.$
\end{proof}
\begin{theor}
Let $R$ be an $\alpha$-skew Armendariz ring. Then $R$ is
$q.\alpha$-skew Armendariz for any epimorphism $\alpha$ of $R.$
\end{theor}
\begin{proof}
Suppose that $R$ is an $\alpha$-skew Armendariz, $p=\sum
\limits_{i=0}^{m}a_{i}x^{i}$ and $q=\sum
\limits_{j=0}^{n}b_{j}x^{j}$ are in $R[x;\alpha ]$, with
$pR[x;\alpha ]q=0. $ Then, in particular, for an arbitrary element
$c$ of $R, $ we get $pcq=0. $ Therefore,
$a_{0}\alpha^{0}(c)\alpha^{0}(b_{j})=a_{0}cb_{j}=0, $ for all
$0\leq j\leq n, $  by using the previous lemma.\\
$pcq=a_{0}c(b_{0}+b_{1}x+...+b_{n}x^{n})+(a_{1}x+...+a_{m}x^{m})c(b_{0}+b_{1}x+...+b_{n}x^{n})=(a_{1}x+...+a_{m}x^{m})c(b_{0}+b_{1}x+...+b_{n}x^{n}).
$

Since $a_{0}cb_{j}=0$ for all $0\leq j\leq n $ we obtain
$0=(a_{1}+a_{2}x+...+a_{m}x^{m-1})xc(b_{0}+b_{1}x+...+b_{n}x^{n})$
$=(a_{1}+a_{2}x+...+a_{m}x^{m-1})\alpha(c)(\alpha(b_{0})x+...+\alpha(b_{n})x^{n+1}).
$\\
 Put $p_{1}=a_{1}+a_{2}x+...+a_{m}x^{m-1}$ and
$q_{1}=\alpha(b_{0})x+...+\alpha(b_{n})x^{n+1}. $

Since $R$ is $\alpha$-skew Armendariz we get $a_{i} \alpha^{i}(c)
\alpha^{i}(b_{j})=0$ for all $1\leq i\leq m, 0\leq j\leq n. $\\
$0=a_{1} \alpha(c)(\alpha(b_{0})x+...+\alpha(b_{n})x^{n+1})+
(a_{2}x+...+a_{m}
x^{m-1})\alpha(c)(\alpha(b_{0})x+...+\alpha(b_{n})x^{n+1})$

$=(a_{2}x+...+a_{m}
x^{m-1})\alpha(c)(\alpha(b_{0})x+...+\alpha(b_{n})x^{n+1})$

$=(a_{2}+a_{3}x+...+a_{m}x^{m-2})\alpha^{2}(c)(\alpha^{2}(b_{0})x^{2}+...+\alpha^{2}(b_{n})x^{n+2}).$
\\
Put $p_{2}=a_{2}+a_{3}x+...+a_{m}x^{m-2},
q_{2}=\alpha^{2}(b_{0})x^{2}+...+\alpha^{2}(b_{n})x^{n+2}. $ Then
if one uses again the fact that $R$ is $\alpha$-skew Armendariz
$p_{2} \alpha^{2}(c) q_{2}=0$ implies $a_{i}
\alpha^{i}(c)\alpha^{i}(b_{j})=0$ for each $2\leq i\leq m, 0\leq
j\leq n.$\\
Continuing this process, we derive $a_{i} R \alpha^{i}(b_{j})=0$
for each $0\leq i\leq m, 0\leq j\leq n. $ Therefore $R$ is
$q.\alpha$-skew Armendariz ring.
\end{proof}
\begin{cl}
Let $R$ be an $\alpha$-Armendariz ring. Then $R$ is
$q.\alpha$-skew Armendariz for any epimorphism $\alpha$ of $R.$
\end{cl}

The previous theorem includes a result of Hong et al. from
\cite{a}, as a particular case.

For an ideal $I$ of $R, $ if $\alpha(I)\subseteq I$ then
 $\bar\alpha : R/I \longrightarrow R/I$ defined by
 $\bar\alpha(a+I)=\alpha(a)+I$ is an endomorphism of $R/I. $\\
 The homomorphic image of a $q.\alpha$-skew Armendariz ring need
 not  be so.
 \begin{ex} We refer to the example from \cite{skew}.
 Consider a ring $R=\left\{\left(
\begin{array}{cc}
a & 0 \\
\bar b & a
\end{array}
\right)| a\in \mathbb{Z}, \bar b \in \mathbb{Z}_{4} \right\}. $
Let $\alpha :R\longrightarrow R$ be the endomorphism defined by
$\alpha\left(\left(
\begin{array}{cc}
a & 0 \\
\bar b & a%
\end{array}%
\right)\right) =\left(
\begin{array}{cc}
a & 0 \\
-\bar b & a
\end{array}%
\right).$\\
The ring $R$ is $q.\alpha$-skew Armendariz. Notice that for
$\left(
\begin{array}{cc}
a & 0 \\
\bar b& a
\end{array}%
\right), \left(
\begin{array}{cc}
c & 0 \\
\bar d& c
\end{array}%
\right) \in R. $ Addition and multiplication are performed as
follows:

 $(a, \bar b)+(c, \bar d)=(a+c, \bar b+ \bar d)$ and $(a,
\bar b).(c, \bar d)=(ac, a\bar d+ \bar b c). $ Let $p=\sum
\limits_{i=0}^{m}\left(
\begin{array}{cc}
a_{i} & 0 \\
\bar b_{i}& a_{i}
\end{array}%
\right)x^{i}=(a_{i}, \bar b_{i}) x^{i}$ and $q=\sum
\limits_{j=0}^{n}\left(
\begin{array}{cc}
c_{j} & 0 \\
\bar d_{j}& c_{j}
\end{array}%
\right)x^{j}=(c_{j}, \bar d_{j}) x^{j} \in R[x;\alpha ]$, with
$pq=0. $\\
$(a_{0}, \bar b_{0})(c_{0}, \bar d_{0})=(a_{0}c_{0}, a_{0} \bar
d_{0}+\bar b_{0}c_{0})=0.$ Therefore, we have $a_{0}c_{0}=0$ and
$a_{0}\bar d_{0}+\bar b_{0} c_{0}=0. $ By using the fact that
$\mathbb{Z}$ is reduced, we obtain $\bar b_{0}c_{0}=0$ and then
$a_{0}\bar d_{0}=0.$ From $(a_{0}, \bar b_{0})(c_{1}, \bar d_{1})+
(a_{1}, \bar b_{1})(c_{0}, -\bar d_{0})=0$ we get
$a_{0}c_{1}+a_{1}c_{0}=0. $ If we multiply on the left side by
$c_{0}, $ then we derive $a_{1}c_{0}=0 $ and $a_{0}c_{1}=0. $
Continuing this process, we have $a_{i}c_{j}=a_{i}\bar d_{j}=\bar
b_{i}c_{j}=0$ for all $0 \leq i \leq m, 0 \leq j \leq n.$ Hence
$\left(
\begin{array}{cc}
a_{i} & 0 \\
\bar b_{i} & a_{i}
\end{array}%
\right).\left(
\begin{array}{cc}
c_{j} & 0 \\
\bar d_{j} & c_{j}
\end{array}%
\right)=0$ for all $i$ and $j. $ Therefore $R$ is $\alpha$-skew
Armendariz and due to Theorem 2 $R$ is $q.\alpha$-skew Armendariz
ring.

 For the ideal $I=\left\{\left(
\begin{array}{cc}
a & 0 \\
0 & a
\end{array}
\right)| a\in 4\mathbb{Z} \right\}$of $R, $ the quotient ring
$R/I\cong\left\{\left(
\begin{array}{cc}
\bar a & 0 \\
\bar b & \bar a
\end{array}
\right)| \bar a, \bar b \in \mathbb{Z}_{4} \right\}$ is not
$q.\bar \alpha$-skew Armendariz.

Indeed, $( \left(
\begin{array}{cc}
\bar 2 & 0 \\
\bar 0 & \bar 2
\end{array}%
\right)+\left(
\begin{array}{cc}
\bar 2 & 0 \\
\bar 1 & \bar 2
\end{array}%
\right)x)R/I[x;\bar \alpha ]( \left(
\begin{array}{cc}
\bar 2 & 0 \\
\bar 0 & \bar 2
\end{array}%
\right)+\left(
\begin{array}{cc}
\bar 2 & 0 \\
\bar 1 & \bar 2
\end{array}%
\right)x)=0. $

However $\left(
\begin{array}{cc}
\bar 2 & 0 \\
\bar 1 & \bar 2
\end{array}%
\right)(R/I)\bar \alpha(\left(
\begin{array}{cc}
\bar 2 & 0 \\
\bar 0 & \bar 2
\end{array}%
\right))\neq 0.$
 \end{ex}
Here is an example of commutative $\alpha$-Armendariz,
$q.\alpha$-Armendariz and $q.\alpha$-skew Armendariz ring that is
not $\alpha$-rigid.

 Given a ring $R$ and a bimodule $_{R}M_{R}.$
The \emph{trivial extension} of $R$ by $M$ is the ring $T(R,M)=R
\oplus M$ with the usual addition and the multiplication defined
as follows:
$$(r_{1},m_{1})(r_{2},m_{2})=(r_{1}r_{2},r_{1}m_{2}+m_{1}r_{2}).$$\\
\begin{ex}
Let $R=T(\mathbb{Z}, \mathbb{Q})$ be the trivial extension of
$\mathbb{Z}$ by $\mathbb{Q}, $ with automorphism $\alpha~
:R\longrightarrow R$ defined by $\alpha((a,s))=(a,s/2). $  The
ring $R$ is $\alpha$-Armendariz and $\alpha$-skew Armendariz but
is not $\alpha$-rigid (see \cite{a}). We prove that $R$ is
$q.\alpha$-skew Armendariz.

Indeed, Let $p=\sum \limits_{i=0}^{m}A_{i}x^{i}$ and $q=\sum
\limits_{j=0}^{n}B_{j}x^{j}$ be elements of $R[x;\alpha ]$, such
that $pR[x;\alpha ]q=\bar 0, $  where $A_{i}=(a_{i}, b_{i}),
B_{j}=(c_{j}, d_{j}) \in R$ for $0\leq i\leq m, 0\leq j\leq n. $
We claim that $A_{i} R \alpha ^{i}(B_{j})=\bar0.$

Let $p=(p_{0}, p_{1}), q=(q_{0}, q_{1}), $ where $p_{0}=\sum
\limits_{i=0}^{m}a_{i}x^{i}, q_{0}=\sum
\limits_{j=0}^{n}c_{j}x^{j} \in \mathbb{Z}[x]$ and $p_{1}=\sum
\limits_{i=0}^{m}b_{i}x^{i}$ and $q_{1}=\sum
\limits_{j=0}^{n}d_{j}x^{j} \in \mathbb{Q}[x]. $ Let $c=(e,f)$ be
an arbitrary element of $R, $ then $$pcq=(p_{0}, p_{1})(e,
f)(q_{0}, q_{1})=(p_{0}e q_{0}, p_{0}e q_{1}+p_{0}f
q_{0}+p_{1}eq_{0})=\bar0$$
\begin{enumerate}
 \item If $p_{0}=0,$ then $p_{1}eq_{0}=0,$ we get either $p_{1}$ or
 $e$ or $q_{0}$ is zero ($\mathbb{Q}[x]$ is a domain).

 We consider the following cases:
\begin{enumerate}
\item Let $p_{0}=0,$ then $a_{i}=0$ for all $i$ and $p_{1}=0,$
then $b_{i}=0$ for all $i.$ Then, clearly, $A_{i} R \alpha
^{i}(B_{j})=\bar0.$ \item If $e=0$ and  $p_{0}=0$ then
$$A_{i} c \alpha^{i}(B_{j})=(0, b_{i})(0, f)\alpha ^{i}(B_{j})=\bar0.$$
\item Let now $p_{0}=0$ and $q_{0}=0 $ then again
$$A_{i} c \alpha ^{i}(B_{j})=(0, b_{i})(e, f)\alpha^{i}((0,
d_{j}))=(0, b_{i}e)(0, d_{j}/2^{i})=\bar0$$
\end{enumerate}
\item If $e=0, $ one  can easily prove again that $A_{i} R \alpha
^{i}(B_{j})=\bar0.$ \item The case $q_{0}=0$ also is handled
similarly.
 \end{enumerate}
 Therefore $R$ is a $q.\alpha$-skew Armendariz ring.
 Similarly one can easily to show that, $R$ is $q.\alpha$-Armendariz
 ring as well.
 \end{ex}

The argument used in the previous example can be applied to show
that the ring $R=T(\mathbb{Z}, \mathbb{Q}),$ with automorphisms
$\alpha_{m} :R\longrightarrow R$ defined by
$\alpha_{m}((a,s))=(a,s/m), $ ($m$ is nonzero integer) also is
$q.\alpha$-skew Armendariz.
 \begin{ex}
 Let $R_{m}=T(\mathbb{Z},\mathbb{Z}_{m})$ with endomorphism $\alpha~
 :R_{m}\longrightarrow R_{m}$ defined by $\alpha((a, \bar s))=(a,
 -\bar s), $ where $m$ is a  prime number. Then $R_{m}$ is $q.\alpha$-skew Armendariz and $q.\alpha$-Armendariz
 ring utilizing the fact that $\mathbb{Z}_{m}$ is Armendariz ring by
 \cite{rege}.
 \end{ex}
Indeed, let $p=\sum \limits_{i=0}^{t}A_{i}x^{i}$ and $q=\sum
\limits_{j=0}^{n}B_{j}x^{j}$ be elements of $R_{m}[x;\alpha ]$,
such that $pR_{m}[x;\alpha ]q=\bar 0, $  where $A_{i}=(a_{i}, \bar
b_{i}), B_{j}=(c_{j},\bar d_{j}) \in R_{m}$ for $0\leq i\leq t,
0\leq j\leq n. $ Then $p=(p_{0},\bar p_{1}), q=(q_{0},\bar q_{1}),
$ where $p_{0}=\sum \limits_{i=0}^{t}a_{i}x^{i}, q_{0}=\sum
\limits_{j=0}^{n}c_{j}x^{j} \in \mathbb{Z}[x]$ and $\bar
p_{1}=\sum \limits_{i=0}^{t}\bar b_{i}x^{i}$ and $\bar q_{1}=\sum
\limits_{j=0}^{n}\bar d_{j}x^{j} \in \mathbb{Z}_{m}[x]. $ If
$c=(e,\bar f)$ be a arbitrary element of $R_{m}, $ then $pcq=\bar
0.$ Hence $p_{0}eq_{0}=0. $ Since $\mathbb{Z}[x]$ is a domain it
implies that either $p_{0}=0$ or $q_{0}=0$ or $e=0. $ Now going
over the options and using the fact that $\mathbb{Z}_{m}[x]$ is
integral domain (remind that $m$ is prime) it is easy to show that
$R_{m}$ is $q.\alpha$-skew Armendariz and $q.\alpha$-Armendariz
ring.
\begin{ex}
Let $R=\left\{\left(
\begin{array}{cc}
a & b \\
0 & 0
\end{array}
\right)| a,b\in F \right\},$ where $F$ is a field and $\alpha$ is
the endomorphism of $R$ defined by $\alpha\left(\left(
\begin{array}{cc}
a & b \\
0 & 0%
\end{array}%
\right)\right) =\left(
\begin{array}{cc}
a & -b \\
0 & 0
\end{array}%
\right).$  The ring $R$ is $q.\alpha$-Armendariz.
\end{ex}
Truly, let $p=\sum \limits_{i=0}^{m}A_{i}x^{i}$ and $q=\sum
\limits_{j=0}^{n}B_{j}x^{j}$ be in $R[x;\alpha ],$ with
$pR[x;\alpha ]q=0,$ where $A_{i}=\left(
\begin{array}{cc}
a_{i} & b_{i} \\
0 & 0
\end{array}
\right) $ and $B_{j}=\left(
\begin{array}{cc}
c_{j} & d_{j} \\
0 & 0
\end{array}
\right)$ for all $0\leq i\leq m, 0\leq j\leq n. $ Suppose that
$A_{0}\neq 0$ and $B_{0}\neq 0.$ Let $C=\left(
\begin{array}{cc}
e & f \\
0 & 0
\end{array}%
\right)$ be an arbitrary element of $R,$ then according to the
hypotheses $pCq=0.$  We have the following system of equalities:
\begin{eqnarray}
(0) &:&A_{0}CB_{0}=0,
\notag \\
(1) &:&A_{0}CB_{1}+A_{1}\alpha(C) \alpha(B_{0})=0 ,
\notag \\
 &&\vdots   \notag \\
 (m+n) &:&A_{0}CB_{m+n}+A_{1}\alpha(C) \alpha(B_{m+n-1})+...+ A_{m+n} \alpha^{m+n}(C) \alpha^{m+n}(B_{0})=0,
\notag \\
\end{eqnarray}
 From (0), we obtain:
 $$A_{0}CB_{0}=\left(
\begin{array}{cc}
a_{0}ec_{0} & a_{0}ed_{0} \\
0 & 0
\end{array}%
\right)=0.$$ Therefore $a_{0}ec_{0}=a_{0}ed_{0}=0.$ If
$a_{0}\neq0$ and $e \neq0, $ then $c_{0}=d_{0}=0, $ which is
contradiction. Hence $a_{0}=0$ or $e=0. $
\begin{enumerate}
\item If $a_{0}=0$ then we get $A_{0}RB_{j}=0$ for all $0\leq
j\leq n. $ From (1), we have:
$$A_{1}\alpha(C) \alpha(B_{0})=\left(
\begin{array}{cc}
a_{1}ec_{0} & -a_{1}ed_{0} \\
0 & 0
\end{array}%
\right)=0.$$ Then $a_{1}ec_{0}=a_{1}ed_{0}=0, $ hence $a_{1}=0$ or
$e=0.$ If $a_{1}=0$ then $A_{1}RB{j}=0$ for all $0\leq j\leq n. $
Continuing this process we derive $A_{i}RB{j}=0$ for all $0\leq
i\leq m $ and $0\leq j\leq n. $ \item If $e=0$ then
$$A_{i}CB_{j}=\left(
\begin{array}{cc}
a_{i} & b_{i} \\
0 & 0
\end{array}%
\right)\left(
\begin{array}{cc}
0 & f \\
0 & 0
\end{array}%
\right)\left(
\begin{array}{cc}
c_{j} & d_{j} \\
0 & 0
\end{array}%
\right)=0$$
\end{enumerate}
Therefore $R$ is $q.\alpha$-Armendariz ring. The similar
observation shows that $R$ is $q.\alpha$-skew Armendariz
 ring as well.

\begin{ex}
Let $F$ be a field. Consider the rings $R_{1}=\left\{\left(
\begin{array}{cc}
0 & a \\
0 & b
\end{array}
\right)| a,b\in F \right\}$ and $R_{2}=\left\{\left(
\begin{array}{cc}
0 & d \\
0 & 0
\end{array}
\right)| d\in F \right\}$ with the endomorphisms
$\alpha_{1}:R_1\longrightarrow R_1$ and
$\alpha_2:R_2\longrightarrow R_2,$  defined by
$\alpha_{1}\left(\left(
\begin{array}{cc}
0 & a \\
0 & b%
\end{array}%
\right)\right) =\left(
\begin{array}{cc}
0 &-a \\
0 & b
\end{array}%
\right), $ and $\alpha_{2}\left(\left(
\begin{array}{cc}
0 & d \\
0 & 0%
\end{array}%
\right)\right) =\left(
\begin{array}{cc}
0 &-d \\
0 & 0
\end{array}%
\right), $ respectively. Then $R_{1}, R_{2}$ are
$\alpha$-Armendariz due to \cite{a}. Hence they are $\alpha$-skew
Armendariz. Therefore they are $q.\alpha$-Armendariz and
$q.\alpha$-skew Armendariz rings.
\end{ex}
\section{Quasi-Armendariz property on skew polynomial Laurent
series rings} In this section we introduce the concepts Laurent
quasi $\alpha$-skew Armendariz and Laurent quasi $\alpha$-skew
Armendariz power series rings. They are in certain sense
generalizations of $q.\alpha$-Armendariz rings. We extend results
from \cite{w} to the case of Laurent quasi $\alpha$-skew
Armendariz and Laurent quasi $\alpha$-skew Armendariz power series
rings. The results of \cite{w} concerned the extension from
$\alpha$-skew Armendariz rings to the $q.\alpha$-Armendariz rings
case. Throughout, this section, $\alpha$ denotes a ring
automorphism. We start with the following definition.
\begin{defn}
A ring $R$ is said to be \emph{$q.\alpha$-skew Armendariz ring of
Laurent type} (or shortly, \emph{Laurent $q.\alpha$-skew
Armendariz}) if whenever $p=\sum \limits_{i=-m}^{n}a_{i}x^{i}$ and
$q=\sum \limits_{j=-t}^{s}b_{j}x^{j}$ in $R[x, x^{-1}; \alpha ], $
satisfy $pR[x, x^{-1}; \alpha ]q=0$ we have $a_{i}R
\alpha^{i}(b_{j})=0, $ $-m\leq i\leq n$ and $-t\leq j\leq s.$
\end{defn}
 Jokanovi$\acute{c}$ \cite{w}, proved that $R$ is $\alpha$-skew Armendariz
 ring if and only if $R$ is $\alpha$-skew Armendariz of Laurent type.
 Here is more general result concerning this matter.
 \begin{theor}
The following statements are equivalent:
\begin{enumerate}
 \item $R$ is a $q.\alpha$-skew Armendariz ring .
 \item $R$ is a Laurent $q.\alpha$-skew Armendariz ring.
\end{enumerate}
 \begin{proof}
Since $R[x; \alpha ]\subset R[x, x^{-1}; \alpha ]$ it is
sufficient to show that if $R$ is $q.\alpha$-skew Armendariz ring
then $R$ is $q.\alpha$-skew Armendariz ring of Laurent type.
Suppose that $p=\sum \limits_{i=-m}^{n}a_{i}x^{i}$ and $q=\sum
\limits_{j=-t}^{s}b_{j}x^{j}$ be elements of $R[x, x^{-1}; \alpha
], $ such that $pR[x, x^{-1}; \alpha ]q~=~0. $ Obviously  for
arbitrary element $c$ of $R$ we have $pcq=0.$ Consider $x^{l}p$
and $qx^{t}$ which are in $R[x;~\alpha ]$ then $x^{l}pcqx^{t}=0.$
It implies that $\alpha^{l}(a_{i}) \alpha^{l}(c)
\alpha^{i+l}(b_{j})=0$ for $-m\leq i\leq n, -t\leq j\leq s.$
Moreover, $\alpha^{l}(a_{i}c\alpha^{i}(b_{j}))=0, $ since $\alpha$
is automorphism. Hence $a_{i}R\alpha^{i}(b_{j})=0.$ Therefore $R$
is Laurent $q.\alpha$-skew Armendariz ring.
 \end{proof}
 \end{theor}
Recall that, a ring $R$ is called \emph{$q.\alpha$-skew power
series Armendariz ring}, if for every $p=\sum
\limits_{i=0}^{\infty}a_{i}x^{i}$ and $q=\sum
\limits_{j=0}^{\infty}b_{j}x^{j} \in R[[x; \alpha ]],$ the
condition $pR[[x; \alpha ]]q=0 $ implies $a_{i}R
\alpha^{i}(b_{j})=0$ for all $0\leq i \leq \infty, 0\leq j \leq
\infty.$
 \begin{defn}
 A ring $R$ is said to be a \emph{$q.\alpha$-skew power series Armendariz ring of
 Laurent type, } if for every $p=\sum \limits_{i=-n}^{\infty}a_{i}x^{i}$ and $q=\sum
\limits_{j=-m}^{\infty}b_{j}x^{j}$ in $R[[x, x^{-1}; \alpha ]], $
satisfy $pR[[x, x^{-1}; \alpha ]]q~=~0$ implies $a_{i}R
\alpha^{i}(b_{j})~=~0,$ for all $-n\leq i \leq \infty$ and $-m\leq
j \leq \infty.$
 \end{defn}
  \begin{theor}
 The following conditions are equivalent:
  \begin{enumerate}
 \item $R$ is a $q.\alpha$-skew power series Armendariz ring.
 \item $R$ is a $q.\alpha$-skew power series Armendariz ring of
 Laurent type.
 \end{enumerate}
\end{theor}
\begin{proof}
The proof is similar to that of Theorem 3.
\end{proof}

\end{document}